\documentclass[twoside,leqno,10pt, A4]{amsart}
\usepackage{amsfonts}
\usepackage{amsmath}
\usepackage{amscd}
\usepackage{amssymb}
\usepackage{amsthm}
\usepackage{amsrefs}
\usepackage{latexsym}
\usepackage{mathrsfs}
\usepackage{bbm}
\usepackage{amscd}
\usepackage{amssymb}
\usepackage{amsthm}
\usepackage{amsrefs}
\usepackage{latexsym}
\usepackage{mathrsfs}
\usepackage{bbm}
\usepackage{enumerate}
\usepackage{graphicx}
\usepackage{color}
\begin{document}

\newtheorem{theorem}[subsection]{Theorem}
\newtheorem{proposition}[subsection]{Proposition}
\newtheorem{lemma}[subsection]{Lemma}
\newtheorem{corollary}[subsection]{Corollary}
\newtheorem{conjecture}[subsection]{Conjecture}
\newtheorem{prop}[subsection]{Proposition}
\numberwithin{equation}{section}
\newcommand{\mr}{\ensuremath{\mathbb R}}
\newcommand{\dif}{\mathrm{d}}
\newcommand{\intz}{\mathbb{Z}}
\newcommand{\ratq}{\mathbb{Q}}
\newcommand{\natn}{\mathbb{N}}
\newcommand{\comc}{\mathbb{C}}
\newcommand{\rear}{\mathbb{R}}
\newcommand{\prip}{\mathbb{P}}
\newcommand{\uph}{\mathbb{H}}
\newcommand{\fief}{\mathbb{F}}
\newcommand{\majorarc}{\mathfrak{M}}
\newcommand{\minorarc}{\mathfrak{m}}
\newcommand{\sings}{\mathfrak{S}}
\newcommand{\fA}{\ensuremath{\mathfrak A}}
\newcommand{\mn}{\ensuremath{\mathbb N}}
\newcommand{\mq}{\ensuremath{\mathbb Q}}
\newcommand{\half}{\tfrac{1}{2}}
\newcommand{\f}{f\times \chi}
\newcommand{\summ}{\mathop{{\sum}^{\star}}}
\newcommand{\chiq}{\chi \bmod q}
\newcommand{\chidb}{\chi \bmod db}
\newcommand{\chid}{\chi \bmod d}
\newcommand{\sym}{\text{sym}^2}
\newcommand{\hhalf}{\tfrac{1}{2}}
\newcommand{\sumstar}{\sideset{}{^*}\sum}
\newcommand{\sumprime}{\sideset{}{'}\sum}
\newcommand{\sumprimeprime}{\sideset{}{''}\sum}
\newcommand{\shortmod}{\ensuremath{\negthickspace \negthickspace \negthickspace \pmod}}
\newcommand{\V}{V\left(\frac{nm}{q^2}\right)}
\newcommand{\sumi}{\mathop{{\sum}^{\dagger}}}
\newcommand{\mz}{\ensuremath{\mathbb Z}}
\newcommand{\leg}[2]{\left(\frac{#1}{#2}\right)}
\newcommand{\muK}{\mu_{\omega}}

\title[Moments and one level density of {H}ecke {$L$}-functions with sextic characters]{Moments and One level density of sextic {H}ecke {$L$}-functions}
%%\date{\today}
\author{Peng Gao and Liangyi Zhao}

\begin{abstract}
Let $\omega = \exp (2\pi i/3)$.  In this paper, we study moments of central values of sextic Hecke $L$-functions of $\mq(\omega)$ and one level density result for the low-lying zeros of sextic Hecke $L$-functions of $\mq(\omega)$. As a corollary, we deduce that, assuming GRH, at least $2/45$ of the members of the sextic family  do not vanish at $s=1/2$.
\end{abstract}

\maketitle


\section{Introduction}

     Due to important arithmetic information encoded by the central values of $L$-functions, their possible zeros there have been investigated extensively. In general, one expects an $L$-function to be non-vanishing at the central point unless there is a good reason for it. This can be further expounded by instancing an $L$-function attached to an elliptic curve of positive rank in view of the Birch-Swinnerton-Dyer conjecture or an $L$-function whose functional equation has the sign of $-1$. In the classical case of Dirichlet $L$-functions, it is conjectured by S. Chowla \cite{chow} that $L(1/2, \chi) \neq 0$ for every primitive Dirichlet character $\chi$. \newline

     There are two methods of studying these potential central zeros of $L$-functions.  One typical approach to achieve the non-vanishing result is to study the moments of a family of $L$-functions. In this way, H. Iwaniec and P. Sarnak \cite{I&S} showed that $L(1/2,\chi) \neq 0$ for at least $1/3$ of the primitive Dirichlet characters $\chi \mod{q}$. M. Jutilia \cite{Jutila} showed that that there are $\gg X/\log X$ fundamental discriminants $d$ with $|d| < X$ such that $L(1/2,\chi_d)\neq 0$. This result is improved by K. Soundararajan \cite{sound1}, who showed that $L(1/2,\chi_{8d})\neq 0$ for at least $87.5\%$ of the real characters $\chi_{8d}$. For the family of $L$-functions of primitive cubic Dirichlet characters, S. Baier and M. P. Young \cite{B&Y} showed that for any $\varepsilon >0$, there are $\gg X^{6/7-\varepsilon}$ primitive cubic Dirichlet characters $\chi$ with conductor $\leq X$ such that $L(1/2,\chi)\neq 0$. \newline

   The result of Baier and Young \cite{B&Y} is similar to another one of W. Luo \cite{Luo}, who studied moments of the family of $L(1/2, \chi)$ associated to cubic Hecke characters $\chi$ with square-free modulus in $\mq(\omega)$ with $\omega=\exp\left( 2\pi i/3 \right)$ and derived a non-vanishing result for this family.  A more general result for the first moment of central $L$-values associated to the $n$-th order Hecke characters of any number field containing the $n$-th root of unity was obtained by S. Friedberg, J. Hoffstein and D. Lieman \cite{FHL}.  Corresponding non-vanishing results were obtained by V. Blomer, L. Goldmakher and B. Louvel in \cite{BGL}. \newline

   In \cite{G&Zhao1, G&Zhao20}, the authors studied moments of $L(1/2,  \chi)$ for families of quadratic and quartic Hecke characters in $\mq(i)$ and $\mq(\omega)$. In this paper, we first extend our studies above to the family of $L$-functions associated to sextic Hecke characters in $\mq(\omega)$. We have \newline
\begin{theorem}
\label{firstmoment}
  Let $W: (0, \infty) \rightarrow \mr$ be a smooth, compactly supported function. For $y \rightarrow \infty$ and any $\varepsilon > 0$,
\begin{align*}
%%\label{1stmom}
   \sumstar_{c \equiv 1 \bmod {36}}L\left( \frac{1}{2},
   \chi_c \right)W\left( \frac{N(c)}{y} \right)=A\widehat{W}(1) y +O(y^{\frac {6}{7}+\varepsilon}),
\end{align*}
   where $\chi_c=\leg {\cdot}{c}_6$ is the sextic
residue symbol in $\mq(\omega)$ (defined in Section~\ref{sec2.4})
\begin{align*}
%%\label{1.4}
   \widehat{W}(1)=\int\limits^{\infty}_0 W(x) \ \dif x.
\end{align*}
  Moreover, $A$ is an explicit constant given in \eqref{eq:c} and $\sum^*$ denotes summation over squarefree elements of $\mz[\omega]$.
\end{theorem}

We note here that for $c=1$, $\chi_c$ is the principal character instead of a sextic one.  It follows from \cite[Corollary 1.4]{BGL} that
\begin{align*}
%%\label{secondmom}
   \sumstar_{c \equiv 1 \bmod {36}} \left| L \left( \frac{1}{2},
   \chi_c \right) \right|^2W \left(  \frac{N(c)}{y} \right) \ll_{\varepsilon} y^{1+\varepsilon},
\end{align*}

    From this and Theorem \ref{firstmoment}, we readily deduce, via a standard argument (see \cite{Luo}), the following
\begin{corollary}
\label{cor1}
  For $y \rightarrow \infty$ and any $\varepsilon > 0$, we have
\begin{equation*}
   \# \left\{ c \in \mz[\omega] : c \equiv 1 \pmod {36}, c \; \mbox{square-free}, \; N(c) \leq y, \; L\left( \frac{1}{2}, \chi_c \right) \neq 0 \right\} \gg_{\varepsilon}
   y^{1-\varepsilon}.
\end{equation*}
\end{corollary}

   The main different feature of Theorem \ref{firstmoment} and the result in \cites{FHL} is that moment of our study is over square-free integers.  Our proof of Theorem \ref{firstmoment} is similar to that of the main result in \cite{Luo}, but our choice of the smooth weight follows the treatment in \cite{B&Y}. \newline

   The other approach, alluded to earlier, towards establishing the non-vanishing result is via the study of the $n$-level densities of low-lying zeros of families of $L$-functions.  The density conjecture of N. Katz and P. Sarnak \cites{KS1, K&S} suggests that the distribution of zeros near $1/2$ of a family of $L$-functions is the same as that of eigenvalues near $1$ of a corresponding classical compact group. For the family of quadratic Dirichlet $L$-functions, the density conjecture implies that $L(1/2, \chi) \neq 0$ for almost all such $\chi$.  Assuming GRH, A. E. \"{O}zl\"uk and C. Snyder  \cite{O&S} computed the one level density for this family to show that $L(1/2, \chi_d) \neq 0$ for at least $15/16$ of the fundamental discriminants $|d| \leq X$. Following the computations carried out in \cites{B&F, ILS}, one can improve this percentage to $(19-\cot (1/4))/16$. \newline

   Our next result is on the one level density of low-lying zeros of a family of sextic Hecke $L$-functions in $\mq(\omega)$. In \cite{Gu}, A. M. G\"ulo\u{g}lu studied the one level density of the low-lying zeros of
a family of cubic Hecke $L$-functions in $\ratq (\omega)$. The result was extended by C. David and A. M. G\"ulo\u{g}lu  in \cite{DG} to hold for test functions whose Fourier transforms are supported in $(-13/11, 13/11)$. This allows them to deduce that at least $2/13$ of the corresponding $L$-functions do not vanish at the central point under GRH. \newline

  Our result here is motivated by the above-mentioned works. To state it, we first need to set some notations.  Let
\[ C(X) = \{ c \in \intz[\omega] : (c,6) =1, \; c \; \mbox{squarefree}, \; X \leq N(c) \leq 2X \} . \]
We shall define in Section \ref{sec2.4} the primitive quadratic Kronecker symbol $\chi^{(72c)}$ for $c \in C(X)$.
We denote the non-trivial zeroes of the Hecke $L$-function
   $L(s, \chi^{(72c)})$ by $1/2+i \gamma_{\chi^{(72c)}, j}$.  Without assuming GRH, we order them as
\begin{equation*}
%%\label{zeroorder}
    \ldots \leq
   \Re \gamma_{\chi^{(72c)}, -2} \leq
   \Re \gamma_{\chi^{(72c)}, -1} < 0 \leq \Re \gamma_{\chi^{(72c)}, 1} \leq \Re \gamma_{\chi^{(72c)}, 2} \leq
   \ldots.
\end{equation*}
    We set
\begin{align*}
%%\label{normalroot}
    \tilde{\gamma}_{\chi^{(72c)}, j}= \frac{\gamma_{\chi^{(72c)}, j}}{2 \pi} \log X
\end{align*}
and define for an even Schwartz class function $\phi$,
\begin{equation*}
%%\label{Sdef}
S(\chi^{(72c)}, \phi)=\sum_{j} \phi(\tilde{\gamma}_{\chi^{(72c)}, j}).
\end{equation*}

We further let $\Phi_X(t)$ be a non-negative smooth function supported on $(1,2)$,
    satisfying $\Phi_X(t)=1$ for $t \in (1+1/U, 2-1/U)$ with $U=\log \log X$ and such that
    $\Phi^{(j)}_X(t) \ll_j U^j$ for all integers $j \geq 0$.  We refer the reader to \cite[Chapter 13]{Tu} for the construction of such functions.  Our result is as follows:
\begin{theorem}
\label{sexticmainthm}
Suppose that GRH is true for the Hecke $L$-functions $L(x, \chi^{(72c)})$ discussed above.  Let $\phi(x)$ be an even Schwartz function whose
Fourier transform $\hat{\phi}(u)$ has compact support in $(-45/43, 45/43)$, then
\begin{align}
\label{sexticdensity}
 \lim_{X \rightarrow +\infty}\frac{1}{\# C(X)}\sumstar_{(c, 6)=1}  S(\chi^{(72c)}, \phi)\Phi_X \left( \frac {N(c)}{X} \right)
 = \int\limits_{\mathbb{R}} \phi(x) \dif x.
\end{align}
   Here, as earlier, the ``$*$'' on the sum over $c$ means that the sum is restricted to squarefree elements $c$ of $\mathbb{Z}[\omega]$ .
\end{theorem}

    The left-hand side of \eqref{sexticdensity} represents the one level density of low-lying zeros of the sextic family of Hecke $L$-functions in $\mq(\omega)$.  On the other hand, the right-hand side of \eqref{sexticdensity} shows that, in connection with the random matrix theory (see the discussions in \cite{G&Zhao2}), the family of sextic Hecke $L$-functions of $\mq(\omega)$ is a unitary family. \newline

   Using the argument in the proof of \cite[Corollary 1.4]{G&Zhao4}, we deduce readily a non-vanishing result for the family of sextic Hecke $L$-functions
under our consideration.
 \begin{corollary}
 Suppose that GRH is true for the Hecke $L$-functions $L(x, \chi^{(72c)})$ and that $1/2$ is a zero of $L \left( s, \chi^{(72c)} \right)$ of order $n_c \geq 0$.  As $X \to \infty$,
 \[  \sumstar_{(c, 6)=1} n_c \Phi_X \left( \frac {N(c)}{X} \right) \leq  \left( \frac{43}{45} + o(1) \right) \# C(X). \]
   Moreover, as $X \to \infty$
\[  \# \{ c \in C(X) : L\left( \frac{1}{2}, \chi^{(72c)} \right) \neq 0 \} \geq \left( \frac{2}{45} + o(1) \right) \# C(X) . \]
    \end{corollary}

  We note that, analogue to the situation in $\mq(\omega)$, the authors previously studied in \cite{G&Zhao4, G&Zhao9} the one level densities of the low-lying zeros of families of quadratic Hecke $L$-functions in all imaginary quadratic number fields of class number one as well as a family of quartic Hecke $L$-functions in $\mq(i)$. Our proof of Theorem \ref{sexticmainthm} is similar to those in \cite{Gu, DG} and \cite{G&Zhao4}. A two dimensional Poisson summation over $\mz(\omega)$  is used to treat the character sums over the primes, then we use essentially a result derived from the theorem of S. J. Patterson in \cite{P} to estimate certain sextic Gauss sums at the primes.

\subsection{Notations} The following notations and conventions are used throughout the paper.\\
\noindent We write $\Phi(t)$ for $\Phi_X(t)$. \newline
\noindent $e(z) = \exp (2 \pi i z) = e^{2 \pi i z}$. \newline
$f =O(g)$ or $f \ll g$ means $|f| \leq cg$ for some unspecified
positive constant $c$. \newline
$f =o(g)$ means $\lim_{x \to \infty}f(x)/g(x)=0$. \newline
$\mu_{[\omega]}$ denotes the M\"obius function on $\mz[\omega]$. \newline
$\zeta_{\mq(\omega)}(s)$ stands for the Dedekind zeta function of $\mq(\omega)$.

\section{Preliminaries}
\label{sec 2}
%%----------------------------------------------------------------------------
\subsection{Residue symbols and Kronecker symbol}
\label{sec2.4}
%%----------------------------------------------------------------------------
   It is well-known that $K=\mq(\omega)$ has class number $1$. For $j=2,3,6$, the symbol $(\frac{\cdot}{n})_j$ is the quadratic ($j=2$), cubic ($j=3$) and
   sextic ($j=6$) residue symbol in the ring of integers $\mathcal{O}_K =\mz[\omega]$.  For a prime $\varpi \in \mz[\omega], (\varpi, j)=1$, we define for
   $a \in \mz[\omega]$, $(a, \varpi)=1$ by $\leg{a}{\varpi}_j \equiv
a^{(N(\varpi)-1)/j} \pmod{\varpi}$, with $\leg{a}{\varpi}_j \in <(-\omega)^{6/j}>$, the cyclic group of order $j$ generated by $(-\omega)^{6/j}$. When
$\varpi | a$, we define
$\leg{a}{\varpi}_j =0$.  Then these symbols can be extended
to any composite $n$ with $(N(n), j)=1$ multiplicatively. We further define $\leg {\cdot}{n}_j=1$ when $n$ is a unit in $\mz[\omega]$. We note that we have
$\leg {\cdot}{n}^2_6=\leg {\cdot}{n}_3, \leg {\cdot}{n}^3_6=\leg {\cdot}{n}_2$ whenever $\leg {\cdot}{n}_6$ is defined. \newline

 We say that
$n=a+b\omega$ in $\mz[\omega]$ is primary if $n \equiv \pm 1
\pmod{3}$, which is equivalent to $a \not \equiv 0 \pmod{3}$, and $b \equiv
0 \pmod{3}$ (see \cite[p. 209]{B&Y}). \newline

   Recall that (see \cite[p. 883]{B&Y}) the following cubic reciprocity law holds for two co-prime primary $n, m$ :
\begin{align*}
%%\label{cubicrec}
    \leg {n}{m}_3 =\leg{m}{n}_3.
\end{align*}

  We also have the following supplementary laws for a primary $n=a+b\omega, n \equiv 1 \pmod 3$ (see \cite[Theorem 7.8]{Lemmermeyer}),
\begin{align} \label{cubicsupp}
  \leg {\omega}{n}_3=\omega^{(1-a-b)/3} \qquad \mbox{and} \qquad \leg {1-\omega}{n}_3=\omega^{(a-1)/3}.
\end{align}

  Following the notations in \cite[Section 7.3]{Lemmermeyer}, we say that any primary $n=a+b\omega \in \mz[\omega]$, with $(n, 6)=1$ is $E$-primary if
\begin{align*}
%%\label{cubicE}
   & a+b \equiv 1 \pmod 4, \quad  \text{if} \quad 2 | b,  \\
   & b \equiv 1 \pmod 4, \quad  \text{if} \quad 2 | a,  \nonumber \\
   & a \equiv 3 \pmod 4, \quad  \text{if} \quad 2 \nmid ab.  \nonumber
\end{align*}

    It follows from \cite[Lemma 7.9]{Lemmermeyer} that $n$ is $E$-primary if and only if $n^3=c+d\omega$ with $c, d \in \mz$ such that $6 | d$ and $c+d \equiv 1 \pmod 4$. This implies that products of $E$-primary numbers are again $E$-primary. Note that in $\intz[\omega]$, every ideal co-prime to $6$ has a unique $E$-primary generator.
     Furthermore, the following sextic reciprocity law holds for two $E$-primary, co-prime numbers $n, m \in \mz[\omega]$ :
\begin{align*}
%%\label{quadreciQw}
    \leg {n}{m}_6 =\leg{m}{n}_6(-1)^{((N(n)-1)/2)((N(m)-1)/2)}.
\end{align*}

   We also have the following supplementary laws for an $E$-primary $n=a+b\omega$ with $(n,6)=1$  (see \cite[Theorem 7.10]{Lemmermeyer}):
\begin{align}
\label{2.05}
  \leg {-\omega}{n}_6=(-\omega)^{(N(n)-1)/6}, \qquad \leg {1-\omega}{n}_2=\leg {a}{3}_{\mz} \qquad \mbox{and} \qquad  \hspace{0.1in} \leg {2}{n}_2=\leg
  {2}{N(n)}_{\mz},
\end{align}
   where $\leg {\cdot}{\cdot}_{\mz}$ denotes the Jacobi symbol in $\mz$. \newline

The above discussions allow us to define a sextic Dirichlet character $\chi^{(72c)} \pmod {72c}$ for any element $c \in \mathcal{O}_K, (c, 6)=1$, such that
for any $n \in (\mathcal{O}_K/72c\mathcal{O}_K)^*$,
\begin{align*}
   \chi^{(72c)}(n)=\leg {72c}{n}_6.
\end{align*}

    One deduces from \eqref{cubicsupp}, \eqref{2.05} and the sextic reciprocity that  $\chi^{(72c)}(n)=1$ when $n \equiv 1 \pmod {72c}$. It follows from this that
    $\chi^{(72c)}(n)$ is well-defined.  As $\chi^{(72c)}(n)$ is clearly multiplicative, of order $6$ and trivial on units, it can be regarded as a
    sextic Hecke character modulo $72c$ of trivial infinite type. We write $\chi^{(72c)}$ for this Hecke character as well and we call it the
    Kronecker symbol. Furthermore, if $c$ is square-free, $\chi^{(72c)}$ is non-principal and primitive. To see this, we write $c=u_c \cdot \varpi_1
    \cdots \varpi_k$ with a unit $u_c$ and primes $\varpi_j$. Suppose $\chi^{(72c)}$ is induced by some $\chi$ modulo $c'$ with $\varpi_j \nmid
    c'$, then by the Chinese Remainder Theorem, there exists an $n$ such that $n \equiv 1 \pmod {72c/\varpi_j}$ and $\leg {\varpi_j}{n}_6 \neq 1$. It
    follows from this that $\chi(n)=1$ but $\chi^{(72c)}(n) \neq 1$, a contradiction. Thus, $\chi^{(72c)}$ can only be possibly
    induced by some $\chi$ modulo $36c$ or modulo $8(1-\omega)^3c$. Suppose it is induced by some $\chi$
    modulo $36c$, then by the Chinese Remainder Theorem, there exists an $n$ such that $n \equiv 1 \pmod {9c}$ and $n \equiv 1+4\omega \pmod {8}$, then we have
    $\chi(n)=1$ but $\chi^{(72c)}(n)= \leg {2}{n}_2 \neq 1$ by \eqref{2.05}, a contradiction. Suppose it is induced by some $\chi$
    modulo $8(1-\omega)^3c$, then by the Chinese Remainder Theorem, there exists an $n$ such that $n \equiv 1 \pmod {8c}$ and $n \equiv 1+3(1-\omega) \pmod
    {9}$, then we have
    $\chi(n)=1$ but
    \[ \chi^{(72c)}(n)= \leg {\omega(1-\omega)}{n}^{-1}_3 \neq 1 \]
by \eqref{cubicsupp} (note that we have $3=-\omega^2(1-\omega)^2$), a contradiction. This implies that $\chi^{(72c)}$ is primitive. This also shows that $\chi^{(-72c)}$ is non-principal.

\subsection{Gauss sums}
\label{section:Gauss}

    Suppose that $n \in \mz[\omega]$ with $(n,6)=1$. For $j=2,3,6$, the quadratic ($j=2$), cubic ($j=3$) and sextic ($j=6$)
 Gauss sum is defined by
\begin{equation*}
   g_j(n)=\sum_{x \shortmod{n}} \leg{x}{n}_j \widetilde{e}\leg{x}{n},
\end{equation*}
   where $\widetilde{e}(z) =e \left( (z-\overline{z})/\sqrt{-3} \right)$. \newline

We can infer from its definition that $g_j(1)=1$.  Moreover, the following well-known relation (see \cite[p. 195]{P}) holds for all $n$:
\begin{align}
\label{2.1}
   |g_j(n)|& =\begin{cases}
    \sqrt{N(n)} \qquad & \text{if $n$ is square-free}, \\
     0 \qquad & \text{otherwise}.
    \end{cases}
\end{align}

    More generally, for any $n,r \in \mz[\omega]$, with $(n,6)=1$, we set
\begin{align*}
%% \label{g2g4}
 g_j(r,n) = \sum_{x \bmod{n}} \leg{x}{n}_6 \widetilde{e}\leg{rx}{n}.
\end{align*}

  We need the following properties of $g_j(r,n)$:
\begin{lemma} \label{quarticGausssum}
   For $j=2,3,6$, any prime $\varpi$ satisfying $(\varpi, 6)=1$, we have
\begin{align}
\label{eq:gmult}
 g_j(rs,n) & = \overline{\leg{s}{n}}_j g(r,n), \quad (s,n)=1, \\
\label{2.03}
   g_j(r,n_1 n_2) &=\leg{n_2}{n_1}_j\leg{n_1}{n_2}_jg_j(r, n_1) g_j(r, n_2), \quad (n_1, n_2) = 1, \\
\label{2.04}
g_6(\varpi^k, \varpi^l)& =\begin{cases}
    N(\varpi)^kg_6(\varpi) \qquad & \text{if} \qquad l= k+1, k \equiv 0 \pmod {6},\\
    N(\varpi)^kg_3(\varpi) \qquad & \text{if} \qquad l= k+1, k \equiv 1 \pmod {6},\\
    N(\varpi)^kg_2(\varpi) \qquad & \text{if} \qquad l= k+1, k \equiv 2 \pmod {6},\\
    N(\varpi)^k\leg{-1}{\varpi}_3\overline{g_3}(\varpi) \qquad & \text{if} \qquad l= k+1, k \equiv 3 \pmod {6},\\
    N(\varpi)^k\leg{-1}{\varpi}_6\overline{g_6}(\varpi) \qquad & \text{if} \qquad l= k+1, k \equiv 4 \pmod {6},\\
    -N(\varpi)^k, \qquad & \text{if} \qquad l= k+1, k \equiv 5 \pmod {6},\\
      \varphi(\varpi^l)=\#(\mz[\omega]/(\varpi^l))^* \qquad & \text{if} \qquad  k \geq l, l \equiv 0 \pmod {6},\\
      0 \qquad & \text{otherwise}.
\end{cases}
\end{align}
\end{lemma}
\begin{proof}
Both \eqref{eq:gmult} and \eqref{2.03} follow easily from the definition. For \eqref{2.04}, the case $l \leq k$ is easily verified.  If $l > k$, then
\begin{align}
\label{2.8}
    \sum_{a \bmod {\varpi^{l}}}\leg {a}{\varpi^l}_6\widetilde{e}\left(\frac {\varpi^ka}{\varpi^l}\right) & =  \sum_{b \bmod \varpi}\leg
    {b}{\varpi^l}_6\sum_{c
    \bmod {\varpi^{l-1}}}\widetilde{e}\left(\frac {c\varpi+b}{\varpi^{l-k}}\right) \\
    &=\sum_{b \bmod \varpi}\leg {b}{\varpi^l}_6\widetilde{e}\left(\frac {b}{\varpi^{l-k}}\right)\sum_{c
    \bmod {\varpi^{l-1}}}\widetilde{e}\left(\frac {c\varpi}{\varpi^{l-k}}\right).  \nonumber
\end{align}

If $l \geq k+2$, we write $c = c_1 \varpi^{l-k-1} + c_2$ where $c_1$ varies over a set of
representatives in $\mz[\omega] \pmod{\varpi^k}$ and $c_2$ varies over a set
of representatives in $\mz[\omega] \pmod{\varpi^{l-k-1}}$ to see that
\begin{align*}
   \sum_{c \bmod {\varpi^{l-1}}}\widetilde{e}\left(\frac {c}{\varpi^{l-k-1}}\right)
   = N(\varpi^k)
   \sum_{c_2 \bmod {\varpi^{l-k-1}}}\widetilde{e} \left( \frac{c_2}{\varpi^{l-k-1}} \right)=0,
\end{align*}
    where the last equality follows from \cite[Lemma, p. 197]{He}. This proves the last case when $l \geq k+2$. \newline

It thus remains to deal with the case in which $l=k+1$.  In this case, the right-hand side expression of \eqref{2.8} is
\begin{align*}
    N(\varpi)^{l-1}\sum_{b \bmod \varpi}\leg {b}{\varpi^l}_6\widetilde{e}\left(\frac {b}{\varpi}\right).
\end{align*}
   The expressions in \eqref{2.04} for $g_6(\varpi^k, \varpi^l)$ follow from this, taking into account the definitions of $g_j(\varpi)$ for $j=2, 3$ and $6$. This
   completes the proof of the lemma.
\end{proof}

\subsection{The approximate functional equation}

      Let $c \in \mathcal{O}_K, c \equiv 1 \pmod {36}$ be square-free. It follows from \eqref{2.05} that
$\chi_c =\leg {\cdot}{c}_6$ is trivial on units, it can be regarded as a primitive Hecke character $\pmod {c}$ of trivial infinite type.
The Hecke $L$-function associated with
$\chi_c$ is defined for $\Re(s) > 1$ by
\begin{equation*}
  L(s, \chi_c) = \sum_{0 \neq \mathcal{A} \subset
  \mathcal{O}_K}\chi_c(\mathcal{A})(N(\mathcal{A}))^{-s},
\end{equation*}
  where $\mathcal{A}$ runs over all non-zero integral ideals in $K$ and $N(\mathcal{A})$ is the
norm of $\mathcal{A}$. As shown by E. Hecke, $L(s, \chi_c)$ admits
analytic continuation to an entire function and satisfies a
functional equation.  We refer the reader to \cites{Gu, Luo, G&Zhao1, G&Zhao20} for a more detailed discussion of these Hecke characters and $L$-functions.  \newline

   Let $G(s)$ be any even function which is holomorphic and bounded in the
   strip $-4<\Re(s)<4$ satisfying $G(0)=1$. We have the following expression for $L(1/2+it, \chi_c)$ for $t \in \rear$ (see \cite[Section 2.4]{G&Zhao20}):
\begin{equation} \label{approxfuneq}
\begin{split}
 L \left( \frac{1}{2}+it, \chi_c \right) = \sum_{0 \neq \mathcal{A} \subset
  \mathcal{O}_K} & \frac{\chi_c(\mathcal{A})}{N(\mathcal{A})^{1/2+it}}V_t \left(\frac{2\pi  N(\mathcal{A})}{x} \right) \\
  & + \frac{g_6(c)}{N(c)^{1/2}}\left(\frac {(2\pi)^2}{|D_k|N(c)} \right )^{it} \frac {\Gamma (1/2-it)}{\Gamma (1/2+it)}\sum_{0 \neq \mathcal{A} \subset
  \mathcal{O}_K}\frac{\overline{\chi}_c(\mathcal{A})}{N(\mathcal{A})^{1/2-it}}V_{-t}\left(\frac{2\pi
  N(\mathcal{A})x}{|D_K|N(c)} \right),
     \end{split}
\end{equation}
    where $g_6(c)$ is the Gauss sum defined in
   Section \ref{sec2.4}, $D_K=-3$ is the discriminant of $K$ and
\begin{align*}
%%\label{2.14}
  V_t \left(\xi \right)=\frac {1}{2\pi
   i}\int\limits\limits_{(2)}\frac {\Gamma(s+1/2+it)}{\Gamma (1/2+it)}G(s)\frac
   {\xi^{-s}}{s} \ \dif s.
\end{align*}

    We write $V$ for $V_0$ and note that for a suitable $G(s)$ (for example $G(s)=e^{s^2}$), we have for any $c>0$ (see \cite[Proposition 5.4]{HIEK}):
\begin{align*}
%%\label{2.15}
  V_t \left(\xi \right) \ll \left( 1+\frac{\xi}{1+|t|} \right)^{-c}.
\end{align*}

   On the other hand, when $G(s)=1$, we have (see \cite[Lemma 2.1]{sound1}) for the $j$-th derivative of $V(\xi)$,
\begin{equation} \label{2.07}
      V\left (\xi \right) = 1+O(\xi^{1/2-\epsilon}) \; \mbox{for} \; 0<\xi<1   \quad \mbox{and} \quad V^{(j)}\left (\xi \right) =O(e^{-\xi}) \; \mbox{for}
      \; \xi >0, \; j \geq 0.
\end{equation}

\subsection{Analytic behavior of Dirichlet series associated with sextic Gauss sums}
\label{section: smooth Gauss}
    For any Hecke character $\chi \pmod {36}$ of trivial infinite type, we let
\begin{align}
\label{h}
   h(r,s;\chi)=\sum_{\substack{(n,r)=1 \\ n \text{ $E$-primary} }}\frac {\chi(n)g_6(r,n)}{N(n)^s}.
\end{align}

    The following lemma gives the analytic behavior of $h(r,s;\chi)$ on $\Re(s) >1$.
 \begin{lemma}
\label{lem1} Let $r$ be $E$-primary. The function $h(r,s;\chi)$ has meromorphic continuation to the complex plane. It is holomorphic in the
region $\sigma=\Re(s) > 1$ except possibly for a pole at $s = 7/6$. For any $\varepsilon>0$, letting $\sigma_1 = 3/2+\varepsilon$, then for $\sigma_1 \geq \sigma \geq \sigma_1- 1/2$, $|s- 7/6|>1/12$ and we have
\begin{equation*}
  h(r,s;\chi) \ll N(r)^{(\sigma_1-\sigma+\varepsilon)/2}(1+t^2)^{5(\sigma_1-\sigma+\varepsilon)/2},
\end{equation*}
  where $t=\Im(s)$. Moreover, the residue satisfies the bound
\begin{equation*}
  \text{Res}_{s=7/6} h(r,s;\chi) \ll N(r)^{1/6+\varepsilon}.
\end{equation*}
\end{lemma}

   Lemma \ref{lem1} is shown in the same manner as that of \cite[Lemma 2.5]{G&Zhao1}. We use Lemma~\ref{lem2} to remove the condition $(n,r)=1$ in \eqref{h}, and the result of Lemma~\ref{lem1} then follows from the proof of the Lemma on \cite[p. 200]{P}, taking into account the observation from \eqref{2.04} that $|g_6(\varpi^k, \varpi^l)| \leq N(\varpi)^{k+1/2}$ when $k < l$ for any $E$-primary prime $\varpi$. \newline

   To state the next lemma, we define for $a, c \in \mz[\omega], (ac, 6)=1$,
\begin{align*}
   \psi_a(c)=(-1)^{((N(a)-1)/2)((N(c)-1)/2)}=\leg {(-1)^{((N(a)-1)/2)}}{c}^3_6.
\end{align*}
   It is easy to see that $\psi_a(c)$ is a Hecke character $\pmod {36}$ of trivial infinite type. \newline

   Furthermore, for any square-free $a \in
   \mz[\omega]$, we let $\{\varpi_1, \cdots, \varpi_k \}$ be the set of distinct $E$-primary prime divisors of $a$ and we define for $1 \leq j \leq 4$,
\begin{align*}
  P_j(a)=\prod^{k}_{i=1}\overline{\leg{(a/\prod^{i}_{l=1}\varpi_l)^j}{\varpi^{j+1}_i}_6}, \quad
\Psi_j(a)=\prod^{k-1}_{i=1}\psi_{\prod^i_{l=1}\varpi^{j+1}_l}(\varpi_{i+1}),
\end{align*}
   where we set the empty product to be $1$. In particular, we have $P_j(a)=1$ when $a$ is a unit and $\Psi_j(a)=1$ when $a$ is a unit or a prime.
   As $\leg{\varpi^j_2}{\varpi^{j+1}_1}_6=\leg{\varpi_2}{\varpi_1}^{j(j+1)}_6$ for two distinct $E$-primary
   primes $\varpi_1, \varpi_2$, one checks easily by induction on the number of prime divisors of $a$ and the sextic reciprocity that $P_j(a)$ is independent of the order of
   $\{\varpi_1, \cdots, \varpi_k \}$ when $j \neq 3$. Similarly, as $\leg{\varpi_1}{\varpi_2}_3=\leg{\varpi_2}{\varpi_1}_3, \psi_{\varpi_1}(\varpi_2)=\psi_{\varpi_2}(\varpi_1)$, $P_3(a)$ and $\Psi_j(a), 1 \leq j \leq 4$ are also independent of the order of
   $\{\varpi_1, \cdots, \varpi_k \}$. \newline

   Now we have
\begin{lemma}
\label{lem2} Let $(rf\alpha,6)=1$. Suppose $f , \alpha$ are square-free and $(r, f ) = 1$, and set
\begin{equation*}
  h(r,f,s;\chi)=\sum_{(n,rf)=1}\frac {\chi(n)g_6(r,n)}{N(n)^s},  \quad  h_{\alpha}(r,s;\chi)=\sum_{(n,\alpha)=1}\frac {\chi(n)g_6(r,n)}{N(n)^s}.
\end{equation*}
  Furthermore suppose $r =r_1r^2_2r^3_3r^4_4r^5_5r^6_6, r^*=r_1r^2_2r^3_3r^4_4r^5_5$ where $r_1r_2r_3r_4r_5$ is square-free, and let $r^*_6$ be the product
  of primes dividing $r_6$. Let $1 \leq j \leq 4$ and
\begin{align*}
 & h^*_{\prod^{j-1}_{i=1}r_i}(r^*,s;\chi) \\
 = &\sum_{\substack{ a|r_j \\ a \text{ $E$-primary}}}\mu_{[\omega]}(a)\chi(a)^{j+1}N(a)^{-(j+1)s}\overline{\leg{r^*/a^j}{a^{j+1}}_6} P_j(a)\Psi_j(a) \left( \prod_{\substack{ \varpi| a \\ \varpi \text{ $E$-primary}}}g_6(\varpi^j, \varpi^{j+1})
 \right)\\
 & \hspace*{2in} \times    h_{\prod^{j-1}_{i=1}r_i}(r^*a^{4-2j},s;\psi_{a^{j+1}}\chi),
\end{align*}
  where the empty product is understood to be $1$. Then
\begin{align}
\label{2.11}
\begin{split}
 h(r,f,s;\chi)=&\sum_{a | f}\frac {\mu_{[\omega]}(a)\chi(a)g_6(r,a)}{N(a)^s}h(a^2r,s;\psi_a\chi), \\
  h(r_1r^2_2r^3_3r^4_4r^5_5r^6_6,s;\chi) =& h(r^*, r^*_6,s;\chi), \\
  h(r^*,s;\chi) =& \prod_{\substack{ \varpi|r_5 \\ \varpi \text{ $E$-primary}}}(1-\chi(\varpi)^6N(\varpi)^{5-6s})^{-1}h_{r_1r_2r_3r_4}(r^*,s;\chi), \\
   h_{\prod^j_{i=1}r_i}(r^*,s;\chi) & \\
  =& \prod_{\substack{ \varpi|r_j \\ \varpi \text{ $E$-primary}}}\left (1-\psi_{\varpi^{j+1}}(\varpi^{5-j})\chi(\varpi)^6N(\varpi)^{-6s}g_6(\varpi^j, \varpi^{j+1})g_6(\varpi^{4-j}, \varpi^{5-j})\right
  )^{-1}  h^*_{\prod^{j-1}_{i=1}r_i}(r^*,s;\chi).
\end{split}
\end{align}
\end{lemma}
\begin{proof}
    As the proof is similar to that of \cite[Lemma 3.6]{B&Y}, we only give the proof for the last equality given in \eqref{2.11} here.
To prove this, we let $a=\prod^5_{i=1}a^i_i \in \mz[\omega]$ and let $\varpi$ be an $E$-primary prime in $\mz[\omega]$ such that $a^*\varpi$ is square-free, where
$a^*=\prod^5_{i=1}a_i$. Then
\begin{align*}
  h_{a^*\varpi}(a\varpi^j,s;\chi) &=\sum_{(n,a^*\varpi)=1}\frac {\chi(n)g_6(a\varpi^j,n)}{N(n)^s} \\
  &=\sum_{(n,a^*)=1}\frac {\chi(n)g_6(a\varpi^j,n)}{N(n)^s}-\sum_{\substack{(n, a^*)=1 \\ \varpi | n }}\frac {\chi(n)g_6(a\varpi^j,n)}{N(n)^s}.
\end{align*}
   Writing in the latter sum $n = \varpi^h n'$ with $(n', \varpi) = 1$, then \eqref{2.03} gives that
\begin{align*}
   g_6(a \varpi^j, \varpi^h n') = \leg{\varpi^h}{n'}_6\leg{n'}{\varpi^h}_6g_6(a \varpi^j, \varpi^h)g_6(a \varpi^j, n').
\end{align*}
    Using \eqref{eq:gmult} and \eqref{2.04} we see that $g_6(a \varpi^j, \varpi^h)=0$ unless $h=j+1$, in which case we deduce from sextic reciprocity and
    \eqref{eq:gmult} that
\begin{align*}
   g_6(a \varpi^j, \varpi^{j+1}n') = g_6(\varpi^j, \varpi^{j+1})\overline{\leg{a}{\varpi^{j+1}}_6}\psi_{\varpi^{j+1}}(n')g_6(a\varpi^{4-j}, n').
\end{align*}
   This implies that
\begin{equation}
\begin{split}
\label{2.16}
  h_{a^*\varpi} & (a\varpi^j,s;\chi)  \\
%  &=\sum_{(n,a^*)=1}\frac {\chi(n)g_6(a\varpi^j,n)}{N(n)^s}-\chi(\varpi^{j+1})N(\varpi)^{-(j+1)s}g_6(\varpi^j, \varpi^{j+1})\overline{\leg{a}{\varpi^{j+1}}_6} h_{a^*\varpi}(a\varpi^{4-j}  ,s;\psi_{\varpi^{j+1}}\chi) \\
  &=h_{a^*}(a\varpi^j,s;\chi)-\chi(\varpi^{j+1})N(\varpi)^{-(j+1)s}g_6(\varpi^j, \varpi^{j+1})\overline{\leg{a}{\varpi^{j+1}}_6}  h_{a^*\varpi}(a\varpi^{4-j}
  ,s;\psi_{\varpi^{j+1}}\chi).
  \end{split}
\end{equation}

    On the other hand,
\begin{align*}
   h_{a^*\varpi}(r^*\varpi^{4-j}  ,s;\psi_{\varpi^{j+1}}\chi) &=\sum_{(n,a^*\varpi)=1}\frac {\psi_{\varpi^{j+1}}(n)\chi(n)g_6(a\varpi^{4-j} ,n)}{N(n)^s} \\
  &=\sum_{(n,a^*)=1}\frac {\psi_{\varpi^{j+1}}(n)\chi(n)g_6(a\varpi^{4-j} ,n)}{N(n)^s}-\sum_{\substack{(n,a^*)=1 \\ \varpi | n }}\frac
  {\psi_{\varpi^{j+1}}(n)\chi(n)g_6(a\varpi^{4-j} ,n)}{N(n)^s}.
\end{align*}
   Again, writing in the latter sum $n = \varpi^hn'$ with $(n', \varpi) = 1$, then \eqref{2.03} yields that
\begin{align*}
   g_6(a\varpi^{4-j} ,\varpi^h n') = \leg{\varpi^h}{n'}_6\leg{n'}{\varpi^h}_6g_6(a\varpi^{4-j} ,\varpi^h)g_6(a\varpi^{4-j}, n').
\end{align*}
Using the same treatment as earlier, \eqref{eq:gmult} and \eqref{2.04} imply that $g(a\varpi^{4-j} ,\varpi^h)=0$ unless $h=5-j$.  Then in that case the sextic reciprocity, together with \eqref{eq:gmult}, yields that
\begin{align*}
   g_6(a\varpi^{4-j} ,\varpi^{5-j} n')= g_6(\varpi^{4-j}, \varpi^{5-j})\overline{\leg{a}{\varpi^{5-j}}_6}\psi_{\varpi^{5-j}}(n')g_6(a\varpi^j, n').
\end{align*}
   This implies that
\begin{equation}
\begin{split}
\label{2.17}
 h_{a^*\varpi} &(a\varpi^{4-j}  ,s;\psi_{\varpi^{j+1}}\chi) \\
  &= h_{a^*}(a\varpi^{4-j}  ,s;\psi_{\varpi^{j+1}}\chi)-\psi_{\varpi^{j+1}}(\varpi^{5-j})\chi(\varpi^{5-j})N(\varpi)^{-(5-j)s}g_6(\varpi^{4-j},
  \varpi^{5-j})\overline{\leg{a}{\varpi^{5-j}}_6}  h_{a^*\varpi}(a\varpi^{j}  ,s;\chi), \nonumber
\end{split}
\end{equation}
   as one checks easily that $\psi_{\varpi^{j+1}}\psi_{\varpi^{5-j}}$ is principal. \newline

Combining \eqref{2.16} and \eqref{2.17}, we get
\begin{align*}
  h_{a^*\varpi} & (a\varpi^j,s;\chi) = (1-\psi_{\varpi^{j+1}}(\varpi^{5-j})\chi(\varpi)^6N(\varpi)^{-6s}g_6(\varpi^j, \varpi^{j+1})g_6(\varpi^{4-j},
  \varpi^{5-j}))^{-1}(h_{a^*}(a\varpi^j,s;\chi) \\
  & \hspace*{1.5in} -\chi(\varpi^{j+1})N(\varpi)^{-(j+1)s}g_6(\varpi^j, \varpi^{j+1})\overline{\leg{a}{\varpi^{j+1}}_6}    h_{a^*}(a\varpi^{4-j}
  ,s;\psi_{\varpi^{j+1}}\chi)).
\end{align*}
   Note that when $j \neq 3$, we have $\leg {\varpi'^{4-j}}{\varpi^{j+1}}_6=1$ for two distinct primes $\varpi, \varpi'$ as $(4-j)(j+1) \equiv 0 \pmod 6$.
When $j=3$, we have $\leg{a}{\varpi^{4}}_6=1$ when $a$ is a cube. These observations together with an induction argument on the number of prime divisors of $a_j$ leads to the last equality given in \eqref{2.11}.
\end{proof}

\subsection{The large sieve with sextic symbols}  One important input of this paper is the following large sieve inequality for sextic Hecke
characters.   The study of the large sieve inequality for characters of a fixed order has a long history.  We refer the reader to \cites{DRHB, DRHB1,
G&Zhao, B&Y, BGL}.

\begin{lemma}\cite[Theorem 1.3]{BGL} \label{sexticls}
Let $M,N$ be positive integers, and $(a_n)$ be an arbitrary sequence of complex numbers, where $n$ runs over $\mz[\omega]$. Then we
have
\begin{equation*}
\label{eq:sextic}
 \sumstar_{\substack{m \in \mz[\omega] \\N(m) \leq M}} \left| \ \sumstar_{\substack{n \in \mz[\omega] \\N(n) \leq N}} a_n \leg{n}{m}_6 \right|^2
 \ll_{\varepsilon} (M + N + (MN)^{2/3})(MN)^{\varepsilon} \sum_{N(n) \leq N} |a_n|^2,
\end{equation*}
   for any $\varepsilon > 0$, where the asterisks indicate that $m$ and $n$ run over square-free $E$-primary elements of $\mz[\omega]$ and $\leg
{\cdot}{m}_6$ is the sextic residue symbol.
\end{lemma}

%%------------------------------------------------------------------------------
\subsection{Poisson Summation}
%%------------------------------------------------------------------------------
The proof of Theorem \ref{sexticmainthm} requires the following Poisson summation formula.
\begin{prop}
\label{Poissonsum6free} Let $n \in \mz[\omega]$ be $E$-primary and $\leg {\cdot}{n}_6$ be the sextic residue symbol $\pmod {n}$. For any Schwartz class function
$\Phi$, we have
\begin{align*}
   \sum_{\substack {c \in \mz[\omega] \\ (c,6)=1}}\leg {c}{n}_6\Phi \left(\frac {N(c)}{X}\right)=\sum_{\substack {(m) \\m | 6}}\mu_{[\omega]}(m)\leg {m}{n}_6\frac {X}{N(m)N(n)}\sum_{k \in
   \mz[\omega]}g_6(k,n)\widetilde{\Phi}\left(\sqrt{\frac {N(k)X}{N(m)N(n)}}\right),
\end{align*}
    where
\begin{align*}
   \widetilde{\Phi}(t) &=\int\limits^{\infty}_{-\infty}\int\limits^{\infty}_{-\infty}\Phi(N(x+y\omega))\widetilde{e}\left(- t(x+y\omega)\right)\dif x \dif y, \quad t \geq 0.
\end{align*}
\end{prop}

   The above result is an easy consequence of a variation of \cite[Lemma 2.6]{G&Zhao2019}, we omit its proof here. When $\Phi(t)$ is the one given in Theorem \ref{sexticmainthm},
   we have the following estimations for $\widetilde{\Phi}$ and its derivatives \cite[(2-15)]{G&Zhao9}: $\widetilde{\Phi}(t) \in \mr$ for any $t \geq 0$ and
\begin{align}
\label{bounds}
     \widetilde{\Phi}^{(\mu)}(t) \ll_{j} \min \{ 1, U^{j-1}t^{-j} \}
\end{align}
    for all integers $\mu \geq 0$, $j \geq 1$ and all $t>0$. \newline

\section{Proof of Theorem ~\ref{firstmoment}}

  We derive readily from \eqref{approxfuneq}, the approximate functional equation, with $G(s)=1$, $t=0$ and $x=3N(c)/z$ that
\begin{align*}
%%\label{sum1}
  \sumstar_{c \equiv 1 \bmod {36}}L \left( \frac{1}{2},
   \chi_c \right) W\left( \frac{N(c)}{y} \right) =M_1+M_2,
\end{align*}
   where
\begin{align*}
%%\label{quadapproxfuneq}
   M_1 &= \sumstar_{c \equiv 1 \bmod {36}} \ \sum_{0 \neq \mathcal{A} \subset
  O_K} \frac{\chi_c(\mathcal{A})}{N(\mathcal{A})^{1/2}}V \left(\frac{2\pi N(\mathcal{A})z }{3N(c)} \right)W\left( \frac{N(c)}{y} \right), \quad \mbox{and} \\
 M_2 &=  \sumstar_{c \equiv 1 \bmod {36}} \ \frac {g_6(c)}{N(c)^{1/2}} \sum_{0 \neq \mathcal{A} \subset
  O_K} \frac{\overline{\chi_c}(\mathcal{A})}{N(\mathcal{A})^{1/2}}V \left(\frac{2\pi N(\mathcal{A}) }{z} \right)W\left( \frac{N(c)}{y} \right).
\end{align*}

\subsection{Evaluating $M_1$, the main term}
\label{section:M1}

   Since any integral non-zero ideal $\mathcal{A}$ in $\mz[\omega]$ has a unique generator
$2^{r_1}(1-\omega)^{r_2}a$, with $r_1$, $r_2\in \intz$, $r_1$, $r_2 \geq 0$ and an $E$-primary $a \in \intz[\omega]$, it follows from \eqref{cubicsupp} and \eqref{2.05} that $\chi_{c}(2)=\chi_c(1-\omega)=1$, which implies that
$\chi_{c}(\mathcal{A}) = \chi_c(a)$. \newline

   We now define for $a$ being $E$-primary, $(c, 6)=1$,
\begin{align*}
  \chi^{(a)}(c)=\leg {a}{c}_6.
\end{align*}
   Similar to our discussions in Section \ref{sec 2}, one can check easily that $\chi^{(a)}$ is a Hecke character modulo $36a$ of trivial infinite type. \newline

The above discussions allow us to recast $M_1$ as
\[  M_1=  \sum_{\substack{r_1,r_2 \geq 0 \\  a \text{ $E$-primary}}} \frac{1}{2^{r_1}3^{r_2/2}N(a)^{1/2}}M(r,a),\]
where
\[ M(r,a)= \sumstar_{c \equiv 1 \bmod {36}}\chi^{(a)}(c)V\left( \frac{\pi 2^{2r_1+1}3^{r_2-1}N(a)z}{y} \frac{y}{N(c)} \right ) W\left( \frac{N(c)}{y}
\right). \]
  Now we use M\"obius inversion to detect the square-free condition of $c$, getting
\[   M(r,a)= \sum_{\substack{l \text{ $E$-primary}}}\mu_{[\omega]}(l)\chi^{(a)}(l^2) M(l,r,a), \]
  with
\[ M(l,r,a)= \sum_{\substack{ cl^2 \equiv 1 \bmod {36} \\ c \text{ $E$-primary}}}\chi^{(a)}(c)V \left( \frac{\pi  2^{2r_1+1}3^{r_2-1}N(a)z}{y} \frac{y
}{N(cl^2)} \right )W\left( \frac{N(cl^2)}{y} \right). \]

    By Mellin inversion, we have
\[  V \left( \frac{\pi  2^{2r_1+1}3^{r_2-1}N(a)z}{y} \frac{y}{N(cl^2)} \right ) W\left( \frac{N(cl^2)}{y} \right) = \frac 1{2\pi i}\int\limits_{(2)} \left(
\frac{y}{N(cl^2)} \right)^s \hat{f}(s) \ \dif s,\]
   where
\begin{align*}
  \hat{f}(s)=\int\limits^{\infty}_{0}V \left( \frac{\pi 2^{2r_1+1}3^{r_2-1}N(a)z}{xy} \right) W(x) x^{s-1} \dif x.
\end{align*}

Integrating by parts together with \eqref{2.07} shows $\hat{f}(s)$ is a function satisfying the bound
\begin{align}
\label{3.1}
  \hat{f}(s) \ll (1+|s|)^{-E} \left( 1+\frac{ 2^{2r_1+1}3^{r_2-1}N(a)z}{y} \right)^{-E},
\end{align}
for all $\Re(s) > 0$ and any integer $E>0$. \newline

    With this notation, we have
\begin{align*}
  M(l,r,a) = \frac 1{2\pi i}\int\limits_{(2)}\hat{f}(s) \left( \frac{y}{N(l^2)} \right)^s\sum_{\substack{  cl^2 \equiv 1 \bmod
  {36} \\ c \text{ $E$-primary}}}\frac {\chi^{(a)}(c)}{N(c)^s} \dif s.
\end{align*}

   Recall that for any $c$, the ray
class group $h_{(c)}$ is defined to be $I_{(c)}/P_{(c)}$, where
$I_{(c)} = \{ \mathcal{A} \in I : (\mathcal{A}, (c)) = 1 \}$ and
$P_{(c)} = \{(a) \in P : a \equiv 1 \pmod{c} \}$ with $I$ and $P$
denoting the group of fractional ideals in $K$ and the subgroup of
principal ideals, respectively. We now use the ray class characters to detect the
condition that $cl^2 \equiv 1 \pmod {36}$, getting
\begin{align*}
  M(l,r,a)=\frac {1}{\#h_{(36)}}\sum_{\psi \bmod {36}}\frac {\psi(l^2)}{2\pi
   i}\int\limits\limits_{(2)}\hat{f}(s) \left( \frac{y}{N(l^2)} \right)^s L(s, \psi\chi^{(a)})\dif s,
\end{align*}
   where $\psi$ runs over all ray class characters $\pmod {36}$, $\#h_{(36)}=108$ and
\begin{align*}
   L(s, \psi\chi^{(a)})=\sum_{\mathcal{A} \neq 0} \frac{\psi(\mathcal{A})\chi^{(a)}(\mathcal{A})}{N(\mathcal{A})^s}.
\end{align*}

   We compute $M_1$ by shifting the contour to the half line.  If $\psi\chi^{(a)}$ is principal, the Hecke $L$-function
has a pole at $s = 1$. We set $M_0$ to be the contribution to $M_1$ of these residues, and $M'_1$ to be the
remainder. \newline

We first evaluate $M_0$. Note that $\psi\chi^{(a)}$ is principal if and only if both $\psi$ and $\chi^{(a)}$ are principal. Hence $a$ must be
   a sixth power. We denote $\psi_0$ for the principal ray class character $\pmod {36}$. Then we have
\begin{align*}
   L \left( s, \psi_0\chi^{(a^6)} \right)=\zeta_{\mq(\omega)}(s)\prod_{(\varpi) |(6a)} \left(1-N(\varpi)^{-s} \right).
\end{align*}

   Let $c_0= \sqrt{3}\pi/9$, the residue of $\zeta_{\mq(\omega)}(s)$ at $s=1$.  Then we have
\begin{align*}
  M_0 &=\frac {y}{\#h_{(36)}} \sum_{\substack{r_1,r_2 \geq 0 \\ a \text{ $E$-primary}}}  \frac{\hat{f}(1)}{2^{r_1}3^{r_2/2}N(a)^3}\text{Res}_{s=1}L
  \left( s, \psi_0\chi^{(a^6)} \right) \sum_{\substack{l \text{ $E$-primary}}} \frac{\mu_{[\omega]}(l)\chi^{(a^6)}(l^2)}{N(l^2)}  \\
  &=\frac {c_0y}{\#h_{(36)}\zeta_{\mq(\omega)}(2)}\sum_{\substack{r_1,r_2 \geq 0 \\ a \text{ $E$-primary}}}
  \frac{\hat{f}(1)}{2^{r_1}3^{r_2/2}N(a)^3}\prod_{(\varpi) |
  (6a)}\left( 1-N(\varpi)^{-1} \right)\prod_{(\varpi) | (6a)} \left( 1-N(\varpi)^{-2} \right)^{-1} \\
  &=\frac {c_0y}{\#h_{(36)}\zeta_{\mq(\omega)}(2)} \sum_{\substack{r_1,r_2 \geq 0 \\ a \text{ $E$-primary}}}
  \frac{\hat{f}(1)}{2^{r_1}3^{r_2/2}N(a)^3}\prod_{(\varpi) |
  (6a)} \left( 1+N(\varpi)^{-1} \right)^{-1}.
\end{align*}

 Set
\begin{equation*}
 Z(u) = \sum_{\substack{r_1,r_2 \geq 0 \\ a \text{ $E$-primary}}} \frac{1}{\pi^u2^{r_1+(2r_1+1)u}3^{r_2/2+(r_2-1)u}N(a)^{3+6u}} \prod_{(\varpi) |
  (6a)} \left( 1+N(\varpi)^{-1} \right)^{-1},
\end{equation*}
which is holomorphic and bounded for $\Re(u) \geq -1/3 + \delta > -1/3$. \newline

   Note that using the Mellin convolution formula shows
\begin{align*}
 & \widehat{f}(1) \\
 =& \int\limits^{\infty}_{0}V \left( \frac{\pi 2^{2r_1+1}3^{r_2-1}N(a)^6z}{xy} \right) W(x) \dif x= \frac{1}{2 \pi i} \int\limits_{(1)}
 \leg{y}{\pi 2^{2r_1+1}3^{r_2-1}N(a)^6z}^s \widehat{W}(1+s) \frac{G(s)}{s} \frac {\Gamma(s+1/2)}{\Gamma (1/2)} \dif s,
\end{align*}	
   where
\begin{equation*}
\widehat{W}(s) = \int\limits_0^{\infty} W(x) x^{s-1} dx.
\end{equation*}

Then
\begin{equation*}
 M_0 = \frac {c_0y}{\#h_{(36)}\zeta_{\mq(\omega)}(2)}  \frac{1}{2 \pi i} \int\limits_{(1)} \left(\frac {y}{z} \right )^s Z(s) \widehat{W}(1+s) \frac{G(s)}{s}
 \frac {\Gamma(s+1/2)}{\Gamma (1/2)} \dif s.
\end{equation*}
We move the contour of integration to $-1/3 + \varepsilon$, crossing a pole at $s=0$ only.  The new contour contributes $O(\left(y/z \right
)^{-1/3 + \varepsilon} y)$, while the pole at $s=0$ gives
\begin{equation}
\label{eq:c}
  A y \widehat{W}(1), \quad \text{where} \quad A= \frac {c_0}{\#h_{(36)}\zeta_{\mq(\omega)}(2)} Z(0).
\end{equation}
Note that $Z(u)$ converges absolutely at $u=0$ so it is easy to express $Z(0)$ explicitly as an Euler product, if desired.
We then conclude that
\begin{align}
\label{m0}
  M_0 =  Ay \widehat{W}(1)+O\left( \left(\frac {y}{z} \right )^{-1/3 + \varepsilon} y \right).
\end{align}

\subsection{Estimating $M_1'$, the remainder term}
\label{section:remainderterm}

    To deal with $M'_1$, we bound everything by absolute values and use \eqref{3.1} to get that for any $E>0$,
\begin{equation} \label{3.2}
\begin{split}
   M'_1  \ll  y^{1/2} \sum_{N(l) \ll \sqrt{y}}\frac {1}{N(l)} & \sum_{\psi \bmod {36}} \sum_{\substack{r_1,r_2 \geq 0 \\ a \text{ $E$-primary}}}
   \frac{1}{2^{r_1}3^{r_2/2}N(a)^{1/2}}
   \left( 1+\frac{ 2^{2r_1+1}3^{r_2-1}N(a)z}{y} \right)^{-E} \\
   & \times \int\limits^{\infty}_{-\infty} \left| L\left(\frac 12+it, \psi\chi^{(a)} \right) \right| (1+|t|)^{-E} \dif t.
\end{split}
\end{equation}

We now need the following estimation to bound the sum over $a$:
\begin{align}
\label{3.04}
    \sum_{\substack{N(a) \leq N \\ a \text{ $E$-primary}}}  N(a)^{-1/2} \left| L\left( \frac 12+it, \psi\chi^{(a)} \right) \right| \ll
    (N(1+|t|))^{1/2+\epsilon}.
\end{align}

The proof of \eqref{3.04} is similar to that of \cite[(39)]{B&Y} and we will give a sketch of the arguments. We factor $a$ as $a_1 a_2^2 a_3^3a_4^4a^5_5a^6_6$
  where $a_1a_2a_3a_4a_5$ is square-free.  Then $\psi\chi^{(a)}$ equals $\psi\chi^{(a_1)}
  (\chi^{(a_2)})^2(\chi^{(a_3)})^3(\overline{\chi^{(a_4)}})^2\overline{\chi^{(a_5)}}$ times a principal character.  Here $(\chi^{(a_2)})^2$ (respectively
  $(\overline{\chi^{(a_4)}})^2$) can be regarded as a cubic Hecke character $\pmod{36a_2}$ (respectively $\pmod{36a_4}$) of trivial infinite type and
  $(\chi^{(a_3)})^3$ can be regarded as a quadratic Hecke character $\pmod{36a_3}$ of trivial infinite type. For each fixed $a_i, 2 \leq i \leq 5$, it
  suffices to show that
\begin{align}
\label{3.7}
 \sumstar_{\substack {N(a_1) \leq N_1 \\ (a_1, a_2a_3a_4a_5)=1}} |L(1/2 + it, \psi\chi^{(a_1)}
 (\chi^{(a_2)})^2(\chi^{(a_3)})^3(\overline{\chi^{(a_4)}})^2\overline{\chi^{(a_5)}})|^2 \ll N_1^{1+\varepsilon} N(a_2a_3a_4a_5)^{1/2+\varepsilon}
 (1+|t|)^{1+\varepsilon},
\end{align}
  where the asterisk indicates that $a_1$ runs over $E$-primary square-free elements of $\mz[\omega]$.
  With this bound and the convergence of the sums over $a_2, a_3, a_4, a_5, a_6$, a use of Cauchy's inequality gives \eqref{3.04}. \newline

   Note that as $a_1a_2a_3a_4a_5$ is square-free, the character $\chi^{(a_1)}
   (\chi^{(a_2)})^2(\chi^{(a_3)})^3(\overline{\chi^{(a_4)}})^2\overline{\chi^{(a_5)}}$ is primitive with conductor $f$ satisfying
\begin{align*}
    \frac {a_1a_2a_3a_4a_5}{(6,a_1a_2a_3a_4a_5)} | f \quad \mbox{and} \quad f | 36a_1a_2a_3a_4a_5.
\end{align*}
     We may now further assume that $\psi\chi^{(a_1)} (\chi^{(a_2)})^2(\chi^{(a_3)})^3(\overline{\chi^{(a_4)}})^2\overline{\chi^{(a_5)}}$ is primitive.
     Thus the Hecke $L$-function
      \[ L(s, \psi\chi^{(a_1)} (\chi^{(a_2)})^2(\chi^{(a_3)})^3(\overline{\chi^{(a_4)}})^2\overline{\chi^{(a_5)}}), \]
     viewed as a degree two $L$-function over $\mq$, has analytic conductor $\ll N_1N(a_2a_3a_4a_5) (1+ t^2)$. \newline

    We then apply the approximate functional equation \eqref{approxfuneq} with $G(s)=e^{s^2}$ for Hecke $L$-functions, removing the weight using the Mellin
    transform to reduce the problem of estimating \eqref{3.7} to bounding
\begin{equation*}
 \sumstar_{N(a_1) \leq N_1} \left| \sum_{\substack{N(n) \ll Q \\ n_1 \text{ $E$-primary}}} \frac{\psi\chi^{(a_1)}
 (\chi^{(a_2)})^2(\chi^{(a_3)})^3(\overline{\chi^{(a_4)}})^2\overline{\chi^{(a_5)}}(n)}{N(n)^{1/2 + it}} \right|^2.
\end{equation*}
   Moreover, by  \cite[Proposition 5.4]{HIEK}, we may truncate the
   sum over $n$ so that $Q \ll  (N_1N(a_2a_3a_4a_5)(1+t^2))^{1/2+\varepsilon}$ with a negligibly small error. \newline

 In the inner sum above, writing $n=n_1n^2_2$ with $n_1$, $n_2$ $E$-primary, $n_1$ square-free
 and using the Cauchy-Schwarz inequality, it is enough to estimate
\begin{equation*}
 \sumstar_{N(a_1) \leq N_1} \left| \ \sumstar_{\substack{N(n_1) \ll Q \\ n \text{ $E$-primary}}} \frac{\psi\chi^{(a_1)}
 (\chi^{(a_2)})^2(\chi^{(a_3)})^3(\overline{\chi^{(a_4)}})^2\overline{\chi^{(a_5)}}(n_1)}{N(n_1)^{1/2 + it}} \right|^2
\end{equation*}
 where the asterisk in the inner sum above indicates that $n_1$ runs over square-free elements of $\mz[\omega]$.
 The bound from Lemma \ref{eq:sextic} then gives the desired estimate for \eqref{3.7}. \newline

   We note that we may truncate the sums over $r_1, r_2, a$ so that $ 2^{2r_1+1}3^{r_2-1}N(a)z \ll  y^{1+\varepsilon}$ in \eqref{3.2} with a negligibly
   small error. We now apply \eqref{3.04} with $N=y^{1+\varepsilon}(2^{2r_1+1}3^{r_2-1}N(a)z)^{-1}$ (we may assume that $y$ is large enough) and treat all
   the sums trivially to see that
\begin{align}
\label{M'_1}
   M'_1 \ll y^{1/2}\left (\frac y{z} \right)^{1/2+\epsilon}.
\end{align}

\subsection{Estimating $M_2$}
\label{section:M2}
From the discussions at the beginning of Section \ref{section:M1}, we have
\begin{align*}
 M_2 &=  \sum_{\substack{r_1,r_2 \geq 0 \\  a \text{ $E$-primary}}} \frac{1}{2^{r_1}3^{r_2/2}N(a)^{1/2}}V \left(\frac{\pi 2^{2r_1+1}3^{r_2}N(a) }{z}
 \right)\sum_{c \equiv 1 \bmod {36}} \ \frac {g_6(c)\overline{\chi_c}(a)}{N(c)^{1/2}} W\left( \frac{N(c)}{y} \right).
\end{align*}
  Note that we can drop the restriction $^*$ in sum over $c$ above,  as it follows from \eqref{2.1} that $g_6(c) = 0$ unless $c$ is square-free. \newline

  We further use the ray class characters to detect the condition that $c \equiv 1 \pmod {36}$ to obtain
\begin{align*}
 M_2 &=  \frac {1}{\#h_{(36)}}\sum_{\substack{r_1,r_2 \geq 0 \\  a \text{ $E$-primary}}} \frac{1}{2^{r_1}3^{r_2/2}N(a)^{1/2}}V \left(\frac{\pi
 2^{2r_1+1}3^{r_2}N(a) }{z} \right)\sum_{\psi \bmod {36}}H(a, \psi, y),
\end{align*}
  where
\begin{equation*}
H(a, \psi, y) =\sum_{c \text{ $E$-primary}} \ \frac {\psi(c) g_6(c)\overline{\chi_c}(a)}{N(c)^{1/2}} W\left( \frac{N(c)}{y} \right).
\end{equation*}
We estimate $H$ with the following:
\begin{lemma}
\label{lemma:Hbound}
 For any $E$-primary $a$ and any ray class character $\psi \pmod {36}$, we have
\begin{equation*}
\label{eq:Hbound}
 H(a, \psi, y) \ll y^{1/2 + \varepsilon} N(a)^{1/4} + y^{2/3} N(a)^{1/6 + \varepsilon}.
\end{equation*}
\end{lemma}
\begin{proof}
  Note that the identity \eqref{eq:gmult} implies $g_6(c)\overline{\chi_c}(a)=g_6(a,c)$ for $(a, c) = 1$.  Introducing the Mellin transform of $w$, we get
\begin{equation}
\label{eq:HlX}
 H(a, \psi, y) =
\frac{1}{2 \pi i} \int\limits_{(2)} \widehat{W}(s) y^s h \left( a, \frac{1}{2} + s; \psi \right) \dif s.
\end{equation}

  We move the line of integration in \eqref{eq:HlX} to $\Re(s) = \frac12 + \varepsilon$, crossing a pole at $s =2/3$, which contributes by Lemma \ref{lem1}
\begin{equation*}
\ll y^{2/3} N(a)^{1/6 + \varepsilon}.
\end{equation*}
The main contribution comes from the new line of integration, which by Lemma \ref{lem1} again gives
\begin{equation*}
\ll  y^{1/2+\varepsilon} N(a)^{1/4}.
\end{equation*}
This completes the proof of Lemma \ref{lemma:Hbound}.
\end{proof}

  Now, to estimate $M_2$, we note that we may truncate the sums over $r_1, r_2, a$ so that $ 2^{2r_1+1}3^{r_2-1}N(a) \ll  z^{1+\varepsilon}$ with a
  negligibly small error. By summing trivially over $r_1, r_2, a$, one easily deduces that
\begin{equation} \label{M2est}
 M_2 \ll y^{1/2+\varepsilon} z^{3/4+\varepsilon}+y^{2/3} z^{2/3 + \varepsilon}.
\end{equation}

\subsection{Conclusion }
     Combining \eqref{M2est} with \eqref{m0} and \eqref{M'_1}, we obtain
\begin{equation*}
 M = cy \widehat{W}(1)+O\left( y^{1/2}\left (\frac y{z} \right)^{1/2+\epsilon}+y^{1/2+\varepsilon} z^{3/4+\varepsilon}+y^{2/3} z^{2/3 + \varepsilon} \right).
\end{equation*}
Setting $z=y^{2/7}$, the proof of Theorem \ref{firstmoment} is completed.
%%--------------------------------------------------------------------------------------------
%%--------------------------------------------------------------------------------------------

%%----------------------------------------------------------------------------
\section{Proof of Theorem \ref{sexticmainthm}}
\label{Section 3}
%%----------------------------------------------------------------------------

The proof of Theorem \ref{sexticmainthm} is similar to that of \cite[Theorem 1.3]{G&Zhao4} and \cite[Theorem 1.1]{G&Zhao9}.  First, as in \cite[Section 3.1]{G&Zhao20}, we can show that $\#C(X) \sim cX$ for some constant $c$ as $X \rightarrow \infty$.
   Next, we take $Z=\log^5 X $ and write
     $\mu_{[\omega]}^2(c)=M_Z(c)+R_Z(c)$ where
\begin{equation*}
    M_Z(c)=\sum_{\substack {(l) , \ l^2|c \\ N(l) \leq Z}}\mu_{[\omega]}(l) \; \quad \mbox{and} \; \quad  R_Z(c)=\sum_{\substack {(l), \  l^2|c \\ N(l) >
    Z}}\mu_{[\omega]}(l).
\end{equation*}

  We shall write $\Phi(t)$ for $\Phi_X(t)$ throughout. We define $S(X,Y; \hat{\phi}, \Phi)=S_M(X,Y; \hat{\phi}, \Phi)+S_R(X,Y; \hat{\phi}, \Phi)$ with
\[ S_M(X,Y; \hat{\phi}, \Phi) =\sum_{(c, 6)=1}M_Z(c) \sum_{\substack{ N(\varpi) \leq Y \\ \varpi \text{ $E$-primary}}} \frac {\log N(\varpi)}{\sqrt{N(\varpi)}}\leg {72c}{\varpi}  \hat{\phi} \left( \frac {\log N(
   \varpi)}{\log X} \right) \Phi\left( \frac {N(c)}{X} \right),\]
    and
\[ S_R(X,Y; \hat{\phi}, \Phi)
=\sum_{(c, 6)=1}R_Z(c) \sum_{\substack{ N(\varpi) \leq Y \\ \varpi \text{ $E$-primary}}} \frac {\log N(\varpi)}{\sqrt{N(\varpi)}}\leg {72c}{\varpi}\hat{\phi} \left( \frac {\log N(
   \varpi)}{\log X} \right) \Phi\left( \frac {N(c)}{X} \right). \]
  Here $\hat{\phi}(u)$ is smooth and has its support in the interval $(-45/43+\varepsilon, 45/43-\varepsilon)$ for some $0<\varepsilon<1$. To emphasize
    this condition, we shall set $Y=X^{45/43-\varepsilon}$ and write the condition $N(\varpi) \leq Y$ explicitly throughout this section . \newline

   Analogue to what is shown in the proof of \cite[Theorem 1.2]{G&Zhao4}, we see that in order to establish Theorem
\ref{sexticmainthm}, it suffices to show that
\begin{align*}
  \lim_{X \rightarrow \infty} \frac{S(X, Y;\hat{\phi}, \Phi)}{X \log X}=0.
\end{align*}

Using standard techniques (see \cite[Section 3.3]{G&Zhao4}), we have that
\begin{equation} \label{error1}
    S_R(X,Y; \hat{\phi}, \Phi)=o(X\log X), \quad \mbox{as} \quad X \rightarrow \infty .
\end{equation}
Indeed, with the truth of GRH, the inner-most sum of $S_R(X,Y; \hat{\phi}, \Phi)$, a character sum over primes, can be bounded very sharply (see \cite[Lemma 2.5]{G&Zhao4}) and we immediately get the estimate in \eqref{error1}. \newline

   To bound $S_M(X,Y; \hat{\phi}, \Phi)$, we rewrite it as
\begin{align*}
%%\label{4.1}
   S_{M} & (X,Y; \hat{\phi}, \Phi)  \\
    =& \sum_{\substack{ N(\varpi) \leq Y\\ \varpi \text{ $E$-primary} }} \frac {\log N(\varpi)}{\sqrt{N(\varpi)}} \leg{72}{\varpi}_6\hat{\phi} \left( \frac {\log N(
   \varpi)}{\log X} \right) \sum_{\substack{N(l) \leq Z \\  l \text{ $E$-primary}}} \mu_{[\omega]}(l)\leg {l^2}{\varpi}_6  \sum_{\substack {c \in \mz[\omega]
   \\ (c, 6)=1} } \leg {c}{\varpi}_6 \Phi \left( \frac {N(cl^2)}{X} \right). \nonumber
\end{align*}

    Applying Lemma \ref{Poissonsum6free} and noting that Lemma \ref{quarticGausssum} gives
\begin{align*}
    g_6(k, \varpi)=\overline{\leg {k}{\varpi}}_6 g_6(\varpi),
\end{align*}
   we can recast $S_M(X,Y; \hat{\phi}, \Phi)$ further as
\begin{equation}
\label{4.02}
\begin{split}
 S_{M}  (X,Y; \hat{\phi}, \Phi) =& X \sum_{\substack{ N(l) \leq Z \\  l \text{ $E$-primary} }} \frac {\mu_{[\omega]}(l)}{N(l^2)} \sum_{\substack {(m) \\m | 6}}\frac {\mu_{[\omega]}(m)}{N(m)}  \\
   & \hspace*{1cm} \times \sum_{\substack{ k \in
   \mz[\omega] \\ k \neq 0}}\sum_{\substack{ N(\varpi) \leq Y\\ \varpi \text{ $E$-primary}}}\frac {\log N(\varpi)}{N(\varpi)^{3/2}}\overline{\leg {72^5km^5l^4}{\varpi}}_6 g_6(\varpi)\hat{\phi} \left(
   \frac {\log N(
   \varpi)}{\log X} \right) \widetilde{\Phi}\left(\sqrt{\frac {N(k)X}{N(ml^2\varpi)}}\right).
   \end{split}
\end{equation}

In essentially the same manner (with some minor modifications) as in the proof of \cite[Lemma 4.2]{G&Zhao4}, we derive from Lemma~\ref{lem1} the following:
\begin{lemma}
\label{lem3} Let $(b, 6)=1$. For any $d  \in \mz[\omega]$, we have
\begin{align*}
 \sum_{\substack {N(c) \leq x \\ c \text{ $E$-primary} \\ c \equiv 0 \bmod {b}}} \overline{\leg {d}{c}}_6 g_6(c)N(c)^{-1/2}
 =    O\left( N(d)^{1/6+\varepsilon}N(b)^{\varepsilon}x^{2/3+\varepsilon}+N(d)^{1/14}N(b)^{-4/7}x^{6/7+\varepsilon} \right).
\end{align*}
\end{lemma}

Using Lemma~\ref{lem3} instead of Proposition 1 of \cite[p. 198]{P} in the sieve identity in Section 4 of \cite{P} and noting that in our case Proposition
2 on \cite[p. 206]{P} is still valid, we obtain that (taking $u_3=X/u_1, u_1=X^{10/(5n+2R)}N(b)^{-5/(5n+2R)}$ as in \cite{P} and noting that we have $n=6,
R=7$ in our case)
\begin{align*}
   E(x;m, k, l) :=\sum_{\substack {N(\varpi) \leq x \\ \varpi \text { $E$-primary} }}\overline{\leg {72^5km^5l^4}{\varpi}}_6 \frac {g_6(\varpi)
   \Lambda(\varpi)}{\sqrt{N(\varpi)}} \ll x^{\varepsilon} \left( N(km^5l^4)^{7/132}x^{59/66}+N(km^5l^4)^{1/44+\varepsilon}x^{1-1/22}\right).
\end{align*}

   It follows from this, \eqref{bounds} and partial summation that
\begin{align*}
\sum_{\substack{ k \in
   \mz[\omega] \\ k \neq 0}} \sum_{\substack {N(\varpi) \leq Y \\ \varpi \text { $E$-primary} }} & \frac {\log N(\varpi)}{N(\varpi)^{3/2}}\overline{\leg {72^5km^5l^4}{\varpi}}_6 g_6(\varpi)\hat{\phi} \left( \frac {\log N(
   \varpi)}{\log X} \right) \widetilde{\Phi}\left(\sqrt{\frac {N(k)X}{N(ml^2\varpi)}}\right) \\
    \ll & \frac {N(m^5l^4)^{7/132+\varepsilon}N(ml^2)^{139/132+\varepsilon}Y^{125/132+\epsilon}U^{3}}{X^{139/132}}+\frac {N(m^5l^4)^{1/44+\varepsilon}N(ml^2)^{45/44+\varepsilon}Y^{43/44+\epsilon}U^{3}}{X^{45/44}}.
\end{align*}

   We then conclude from \eqref{error1} and \eqref{4.02} that
\begin{equation*}
   S(X,Y; {\hat \phi}, \Phi)= o\left( X \log X \right), \quad \mbox{as} \quad X \to \infty.
\end{equation*}
mindful of $U=\log \log X$, $Z=\log^5 X$. This completes the proof of Theorem \ref{sexticmainthm}. \newline

\noindent{\bf Acknowledgments.} P. G. is supported in part by NSFC grant 11871082 and L. Z. by the FRG grant PS43707 at the University of New South Wales (UNSW).  The authors would like to thank the anonymous referee for his/her careful reading of the paper and many helpful comments.

\bibliography{biblio}
\bibliographystyle{amsxport}

\vspace*{.5cm}

\noindent\begin{tabular}{p{6cm}p{6cm}p{6cm}}
School of Mathematical Sciences & School of Mathematics and Statistics \\
Beihang University & University of New South Wales \\
Beijing 100191 China & Sydney NSW 2052 Australia \\
Email: {\tt penggao@buaa.edu.cn} & Email: {\tt l.zhao@unsw.edu.au} \\
\end{tabular}

\end{document}